\documentclass[12pt,a4paper]{amsart}
\usepackage{amsmath}
\usepackage{amssymb}
\usepackage{drpack}
\usepackage[english]{babel}
\usepackage{verbatim}
\usepackage[shortlabels]{enumitem}
\usepackage{color}
\usepackage{tikz-cd}

\newtheoremstyle{mio}%
{}{} 
{\itshape}{} 
{\bfseries}{.}{ } 
{#1 #2\thmnote{~\mdseries(#3)}} 
\theoremstyle{mio}
\newtheorem{teor}{Theorem}[section]
\newtheorem{cor}[teor]{Corollary}
\newtheorem{prop}[teor]{Proposition}
\newtheorem{lemma}[teor]{Lemma}
\newtheorem{defin}[teor]{Definition}

\newtheoremstyle{definition2}%
{}{} 
{}{} 
{\bfseries}{.}{ } 
{#1 #2\thmnote{\mdseries~ #3}} 
\theoremstyle{definition2}
\newtheorem{ex}[teor]{Example}
\newtheorem{oss}[teor]{Remark}

\newcommand{\monprinc}{\mathcal{P}}
\newcommand{\monrad}{\mathrm{Rad}}

\DeclareMathOperator{\rk}{rk}
\newcommand{\Cl}{\mathrm{Cl}}

\newcommand{\fin}{\mathrm{fin}}
\newcommand{\Inv}{\mathrm{Inv}}
\newcommand{\insid}{\mathcal{I}}
\newcommand{\pos}{\mathrm{pos}}
\newcommand{\negt}{\mathrm{neg}}
\newcommand{\insprinc}{\mathrm{Princ}}
\newcommand{\nz}{\bullet}
\newcommand{\VV}{\mathcal{V}}
\newcommand{\mprop}{\mathfrak{M}}

\title[Radicals of principal ideals and the class group]{Radicals of principal ideals and the class group of a Dedekind domain}
\author{Dario Spirito}
\date{\today}
\address{Dipartimento di Matematica e Fisica, Universit\`a degli Studi ``Roma Tre'', Roma, Italy}
\email{spirito@mat.uniroma3.it}
\subjclass[2010]{13F05, 13C20}
\keywords{Dedekind domain; class group; radical ideals; principal ideals}

\begin{document}

\begin{abstract}
For a Dedekind domain $D$, let $\monprinc(D)$ be the set of ideals of $D$ that are radical of a principal ideal. We show that, if $D,D'$ are Dedekind domains and there is an order isomorphism between $\monprinc(D)$ and $\monprinc(D')$, then the rank of the class groups of $D$ and $D'$ is the same. 
\end{abstract}

\maketitle

\section{Introduction}
The class group $\Cl(D)$ of a Dedekind domain $D$ is defined as the quotient between the group of the nonzero fractional ideals of $D$ and the subgroup of the principal ideals of $D$. Since $\Cl(D)$ is trivial if and only if $D$ is a principal ideal domain (equivalently, if and only if it is a unique factorization domain), the class group can be seen as a way to measure how much unique factorization fails in $D$. For this reason, the study of the class group is an important part of the study of Dedekind domains.

It is a non-obvious fact that the class group of $D$ actually depends only on the multiplicative structure of $D$, or, from another point of view, depends only on the set of nonzero principal ideals of $D$. Indeed, the class group of $D^\nz:=D\setminus\{0\}$ as a monoid (where the operation is the product) is isomorphic to the class group of $D$ as a Dedekind domain (see Chapter 2 -- in particular, Section 2.10 -- of \cite{ger-HK}), and thus if $D$ and $D'$ are Dedekind domains whose sets of principal ideals are isomorphic (as monoids) then the class groups of $D$ and $D'$ are isomorphic too.

In this paper, we show that the rank of $\Cl(D)$ can be recovered by considering only the set $\monprinc(D)$ of the ideals that are \emph{radical} of a principal ideal: that is, we show that if $\monprinc(D)$ and $\monprinc(D')$ are isomorphic as partially ordered sets then the ranks of $\Cl(D)$ and $\Cl(D')$ are equal. The proof of this result can be divided into two steps. 

In Section \ref{sect:coprime} we show that an order isomorphism between $\monprinc(D)$ and $\monprinc(D')$ can always be extended to an isomorphism between the sets $\monrad(D)$ and $\monrad(D')$ of all radical ideals of $D$ (Theorem \ref{teor:princ->rad}): this is accomplished by considering these sets as (non-cancellative) semigroups and characterizing coprimality in $D$ through a version of coprimality in $\monprinc(D)$ (Proposition \ref{prop:prod-coprime}).

In Section \ref{sect:ordini} we link the structure of $\monprinc(D)$ and $\monrad(D)$ with the structure of the tensor product $\Cl(D)\otimes\insQ$ as an ordered topological space; in particular, we interpret the set of inverses of a set $\Delta\subseteq\Max(D)$ with respect to $\monprinc(D)$ (see Definition \ref{def:InvP}) as the negative cone generated by the images of $\Delta$ in $\Cl(D)\otimes\insQ$ (Proposition \ref{prop:InvP}) and use this connection to calculate the rank of $\Cl(D)$ in function of some particular partitions of an ``inverse basis'' of $\Max(D)$ (Propositions \ref{prop:reaypart} and \ref{prop:corrispreay}). As this construction is invariant with respect to isomorphism, we get the main theorem (Theorem \ref{teor:rank}).

In Section \ref{sect:counterex} we give three examples, showing that some natural extensions of the main result do not hold.

\section{Notation and preliminaries}
Throughout the paper, $D$ will denote a \emph{Dedekind domain}, that is, a one-dimensional integrally closed Noetherian integral domain; equivalently, a one-dimensional Noetherian domain such that $D_P$ is a discrete valuation ring for all maximal ideals $P$. For general properties about Dedekind domains, the reader may consult, for example, \cite[Chapter 7, \textsection 2]{bourbaki_ac}, \cite[Chapter 9]{atiyah} or \cite[Chapter 1]{neukirch}.

We use $D^\nz$ to indicate the set $D\setminus\{0\}$. We denote by $\Max(D)$ the set of maximal ideals of $D$. If $I$ is an ideal of $D$, we set
\begin{equation*}
V(I):=\{P\in\Spec(D)\mid I\subseteq D\}.
\end{equation*}
If $I=xD$ is a principal ideal, we write $V(x)$ for $V(xD)$. If $I\neq(0)$, the set $V(I)$ is always a finite subset of $\Max(D)$. We denote by $\rad(I)$ the radical of the ideal $I$, and we say that $I$ is a \emph{radical ideal} (or simply that $I$ is \emph{radical}) if $I=\rad(I)$.

Every nonzero proper ideal $I$ of $D$ can be written uniquely as a product $P_1^{e_1}\cdots P_n^{e_n}=P_1^{e_1}\cap\cdots\cap P_n^{e_n}$, where $P_1,\ldots,P_n$ are distinct maximal ideals and $e_1,\ldots,e_n\geq 1$. In particular, in this case we have $V(I)=\{P_1,\ldots,P_n\}$, and $\rad(I)=P_1\cdots P_n$. An ideal is radical if and only if $e_1=\cdots=e_n=1$. If $P$ is a maximal ideal, the \emph{$P$-adic valuation} of an element $x$ is the exponent of $P$ in the factorization of $xD$; we denote it by $v_P(x)$. (If $x\notin P$, i.e., if $P$ does not appear in the factorization, then $v_P(x)=0$.)

If $P_1,\ldots,P_k$ are distinct maximal ideals and $e_1,\ldots,e_k\inN$, then by the approximation theorem for Dedekind domains (see, e.g., \cite[Chapter VII, \textsection 2, Proposition 2]{bourbaki_ac}) there is an element $x\in D$ such that $v_{P_i}(x)=e_i$ for $i=1,\ldots,k$.

A \emph{fractional ideal} of $D$ is a $D$-submodule $I$ of the quotient field $K$ of $D$ such that $xI\subseteq D$ (and thus $xI$ is an ideal of $D$) for some $x\in D^\nz$. The set $\insfracid(D)$ of nonzero fractional ideals of $D$ is a group under multiplication; the inverse of an ideal $I$ is $I^{-1}:=(D:I):=\{x\in K\mid xI\subseteq D\}$. A nonzero fractional ideal $I$ can be written uniquely as $P_1^{e_1}\cdots P_n^{e_n}$, where $P_1,\ldots,P_n$ are distinct maximal ideals and $e_1,\ldots,e_n\inZ\setminus\{0\}$ (with the empty product being equal to $D$). Thus, $\insfracid(D)$ is isomorphic to the free abelian group over $\Max(D)$. The quotient between this group and its subgroup formed by the principal fractional ideals is called the \emph{class group} of $D$, and is denoted by $\Cl(D)$.

For a set $S$, we denote by $\mathfrak{P}_\fin(S)$ the set of all finite and nonempty subsets of $S$.

\section{The two semilattices $\monprinc(D)$ and $\monrad(D)$}\label{sect:coprime}
Let $(X,\leq)$ be a \emph{meet-semilattice}, that is, a partially ordered set where every pair of elements has an infimum. Then, the operation $x\wedge y$ associating to $x$ and $y$ their infimum is associative, commutative and idempotent, and it has a unit if and only if $X$ has a maximum. The order of $X$ can also be recovered from the operation: $x\geq y$ if and only if $x$ divides $y$ in $(X,\wedge)$, that is, if and only if there is a $z\in X$ such that $y=x\wedge z$. A \emph{join-semilattice} is defined in the same way, but using the supremum instead of the infimum.

Let now $D$ be a Dedekind domain. We will be interested in two structures of this kind.

The first one is the semilattice $\monrad(D)$ of all nonzero radical ideals of $D$. In this case, the order $\leq$ is the usual containment order, while the product is equal to
\begin{equation*}
I\wedge J:=I\cap J=\rad(IJ).
\end{equation*}

The second one is the semilattice $\monprinc(D)$ of the ideals of $D$ that are radical of a nonzero, \emph{principal} ideal of $D$. This is a subsemilattice of $\monrad(D)$ since
\begin{equation*}
\rad(aD)\wedge\rad(bD)=\rad(abD),
\end{equation*}
i.e., the product of two elements of $\monprinc(D)$ remains inside $\monprinc(D)$. 

A nonzero radical ideal $I$ is characterized by the finite set $V(I)$. Hence, the map from $\monrad(D)$ to $\mathfrak{P}_\fin(\Max(D))$ sending $I$ to $V(I)$ is an order-reversing isomorphism of partially ordered sets, which becomes an order-reversing isomorphism of semilattices if the operation on the power set is the union. We denote by $\VV(D)$ the image of $\monprinc(D)$ under this isomorphism; that is, $\VV(D):=\{V(x)\mid x\in D^\nz\}$. The inverse of this map is the one sending a set $Z$ to the intersection of the prime ideals contained in $Z$.

Those semilattices have neither an absorbing element (which would be the zero ideal) nor a unit (which should be $D$ itself).

\begin{lemma}
Let $X,Y\in\mathfrak{P}_\fin(\Max(D))$ (resp., $X,Y\in\VV(D)$). Then, $X|Y$ in $\mathfrak{P}_\fin(\Max(D))$ (resp., $X|Y$ in $\VV(D)$) if and only if $X\subseteq Y$.
\end{lemma}
\begin{proof}
If $X|Y$, then $Y=X\cup Z$ for some $Z\in\VV(D)$, and thus $X\subseteq Y$. If $X\subseteq Y$, then $Y=Y\cup X$ and thus $X|Y$. (This works both in $\mathfrak{P}_\fin(\Max(D))$ and in $\VV(D)$.)
\end{proof}

\begin{defin}
Let $M$ be a commutative semigroup. We say that $a_1,\ldots,a_n\in M$ are \emph{product-coprime} if, whenever there is an $x\in M$ such that $x=a_1b_1=a_2b_2=\cdots=a_nb_n$, then for every $j$ the element $a_j$ divides $\prod_{i\neq j} b_i$.
\end{defin}

\begin{prop}\label{prop:prod-coprime}
Let $D$ be a Dedekind domain, and let $a_1,\ldots,a_n\in D^\nz$. Then, $a_1,\ldots,a_n$ are coprime in $D$ if and only if $V(a_1),\ldots,V(a_n)$ are product-coprime in $\VV(D)$.
\end{prop}
\begin{proof}
Suppose that $a_1,\ldots,a_n$ are coprime, and let $X\in\VV(D)$ be such that $X=V(a_1)\cup B_1=\cdots=V(a_1)\cup B_n$ for some $B_1,\ldots,B_n\in\VV(D)$. By symmetry, it is enough to prove that $V(a_1)$ divides $B_2\cup\cdots\cup B_n$ in $\VV(D)$, i.e., that $V(a_1)\subseteq B_2\cup\cdots\cup B_n$. Take any prime ideal $P\in V(a_1)$: since $a_1,\ldots,a_n$ are coprime there is a $j$ such that $P\notin V(a_j)$. However, $P\in V(a_j)\cup B_j$, and thus $P\in B_j$. Therefore, $V(a_i)\subseteq B_2\cup\cdots\cup B_n$, as claimed.

Conversely, suppose $V(a_1),\ldots,V(a_n)$ are product-coprime, and suppose that $a_1,\ldots,a_n$ are not coprime. Then, there is a prime ideal $P$ containing all $a_i$; passing to powers, without loss of generality we can suppose that the $P$-adic valuation of the $a_i$ is the same, say $v_P(a_i)=t$ for every $i$. By prime avoidance, there is a $b_1\in D\setminus P$ such that $v_Q(b_1)\geq v_Q(a_i)$ for all $i>1$ and all $Q\neq P$. Let $x:=a_1b_1$. By construction, $a_i|x$ for each $i$, and thus we can find $b_2,\ldots,b_n\in D$ such that $x=a_ib_i$. Therefore, $V(x)=V(a_i)\cup V(b_i)$ for every $i$; by hypothesis, it follows that $V(a_1)$ divides $V(b_2)\cup\cdots\cup V(b_n)$, i.e., that $V(a_i)\subseteq V(b_2)\cup\cdots\cup V(b_n)$. However, $v_P(x)=v_P(a_1)+v_P(b_1)=t$, and thus $v_P(b_i)=0$ for every $i$; in particular, $P\notin V(b_i)$ for every $i$. This is a contradiction, and thus $a_1,\ldots,a_n$ are coprime.
\end{proof}

\begin{defin}
Let $M$ be a commutative semigroup. We say that $I\subsetneq M$ is \emph{product-proper} if no finite subset of $I$ is product-coprime. We denote the set of maximal product-proper subsets of $M$ by $\mprop(M)$.
\end{defin}

\begin{prop}\label{prop:full}
Let $D$ be a Dedekind domain. The maps
\begin{equation*}
\begin{aligned}
\nu\colon\Max(D) & \longrightarrow\mprop(\VV(D)),\\
P & \longmapsto\{V(x)\mid x\in P\}
\end{aligned}
\end{equation*}
and
\begin{equation*}
\begin{aligned}
\theta\colon\mprop(\VV(D)) & \longrightarrow\Max(D),\\
\mathcal{Y} & \longmapsto\{x\in D\mid V(x)\in \mathcal{Y}\}
\end{aligned}
\end{equation*}
are bijections, inverse one of each other.
\end{prop}
\begin{proof}
We first show that $\nu$ and $\theta$ are well-defined.

If $P$ is a maximal ideal of $D$, then $P\in X$ for every $X\in\nu(P)$; thus, if $V(a)\in\nu(P)$ then $a\in P$ and $\nu(P)$ is product-proper. If $\nu(P)\subsetneq\mathcal{Y}\subseteq\VV(D)$, take $Y\in\mathcal{Y}\setminus\nu(P)$: then, $Y=V(b)$ for some $b\notin P$. If $Y=\{Q_1,\ldots,Q_k\}$, by prime avoidance we can find $a\in P\setminus(Q_1\cup\cdots\cup Q_k)$; then, $a$ and $b$ are coprime and thus $V(a)$ and $V(b)$ are product-coprime. Hence, $\nu(P)$ is a maximal product-proper subset of $\VV(D)$.

Conversely, let $\mathcal{Y}\in\mprop(\VV(D))$. If $\theta(\mathcal{Y})$ is contained in some prime ideal $P$, then $\mathcal{Y}\subseteq\nu(P)$, and thus we must have $\mathcal{Y}=\nu(P)$; in particular, $\theta(\mathcal{Y})=P\in\Max(D)$. If $\theta(\mathcal{Y})$ is not contained in any prime ideal, let $V(a)=\{Q_1,\ldots,Q_k\}\in\mathcal{Y}$. Since $\theta(\mathcal{Y})\nsubseteq Q_i$, for every $i$ we can find $b_i\notin Q_i$ such that $V(a_i)\in\mathcal{Y}$; then, $a,b_1,\ldots,b_n$ are coprime and thus $V(a),V(b_1),\ldots,V(b_n)$ are a product-coprime subset of $\mathcal{Y}$, a contradiction. Hence $\mathcal{Y}=\nu(P)$.

The fact that they are inverses one of each other follows similarly.
\end{proof}

\begin{teor}\label{teor:princ->rad}
Let $D,D'$ be Dedekind domains. If there is an order isomorphism $\psi:\monprinc(D)\longrightarrow\monprinc(D)$, then there is an order isomorphism $\Psi:\monrad(D)\longrightarrow\monrad(D')$ extending $\psi$.
\end{teor}
\begin{proof}
The statement is equivalent to saying that any isomorphism $\phi:\VV(D)\longrightarrow\VV(D')$ can be extended to an isomorphism $\Phi:\mathfrak{P}_\fin(\Max(D))\longrightarrow\mathfrak{P}_\fin(\Max(D'))$. For simplicity, let $\mathfrak{P}:=\mathfrak{P}_\fin(\Max(D))$ and $\mathfrak{P}':=\mathfrak{P}_\fin(\Max(D'))$.

If $\phi$ is an isomorphism, then it sends product-proper sets into product-proper sets, and thus $\phi$ induces a bijective map $\eta_1:\mprop(\VV(D))\longrightarrow\mprop(\VV(D'))$. Using the map $\theta$ of Proposition \ref{prop:full}, $\eta_1$ induces a bijection $\eta:\Max(D)\longrightarrow\Max(D')$, such that the diagram
\begin{equation*}
\begin{tikzcd}
\mprop(\VV(D))\arrow{d}{\eta_1}\arrow{r}{\theta} & \Max(D)\arrow{d}{\eta}\\
\mprop(\VV(D'))\arrow{r}{\theta'} & \Max(D)
\end{tikzcd}
\end{equation*}
commutes (explicitly, $\eta=\theta'\circ\eta_1\circ\theta^{-1}$). In particular, $\eta$ induces an order isomorphism $\Phi$ between $\mathfrak{P}$ and $\mathfrak{P}'$, sending $X\subseteq\Max(D)$ to $\eta(X)\subseteq\Max(D')$. To conclude the proof, we need to show that $\Phi$ extends $\phi$.

Let $X=\{P_1,\ldots,P_k\}\in\VV(D)$. Then, by definition, $\Phi(X)=\eta(X)=\{\eta(P_1),\ldots,\eta(P_k)\}$. The maximal product-proper subsets of $\VV(D)$ containing $X$ are $\mathcal{Y}_i:=\nu(P_i)$, for $i=1,\ldots,k$; since $\phi$ is an isomorphism, the maximal product-proper subsets of $\VV(D')$ containing $\phi(X)$ are the sets $\phi(\mathcal{Y}_i)$. By construction, $\phi(\mathcal{Y}_i)=\eta_1(\mathcal{Y}_i)$; however, $\theta'(\eta_1(\mathcal{Y}_i))=\eta(P_i)$, and thus $\eta(X)=\{\phi(\mathcal{Y}_1),\ldots,\phi(\mathcal{Y}_k)\}=\phi(X)$. Thus, $\Phi$ extends $\phi$, as claimed.
\end{proof}

The following corollary was obtained, with a more \emph{ad hoc} reasoning, in the proof of \cite[Theorem 2.6]{golomb-almcyc}.
\begin{cor}\label{cor:rk0}
Let $D,D'$ be Dedekind domains such that $\monprinc(D)$ and $\monprinc(D')$ are order-isomorphic. Then, $\Cl(D)$ is torsion if and only if $\Cl(D')$ is torsion.
\end{cor}
\begin{proof}
The class group of $D$ is torsion if and only if every prime ideal has a principal power \cite[Theorem 3.1]{gilmer_qr}, and thus if and only if $\monprinc(D)=\monrad(D)$.

If $\monprinc(D)$ and $\monprinc(D')$ are isomorphic, then by Theorem \ref{teor:princ->rad} there is an isomorphism $\Phi:\monrad(D)\longrightarrow\monrad(D')$ sending $\monprinc(D)$ to $\monprinc(D')$; hence, $\monrad(D)=\monprinc(D)$ if and only if $\monrad(D')=\monprinc(D')$. Therefore, $\Cl(D)$ is torsion if and only if $\Cl(D')$ is torsion.
\end{proof}

\begin{oss}
Let $\insprinc(D)$ be the set of principal ideals of $D$ and $\insid(D)$ be the set of all ideals of $D$. 

The method used in this section can also be applied to prove the analogous result for non-radical ideals, i.e., to prove that an isomorphism $\phi:\insprinc(D)\longrightarrow\insprinc(D')$ can be extended to an isomorphism $\Phi:\insid(D)\longrightarrow\insid(D')$.

The most obvious analogue of Proposition \ref{prop:prod-coprime} does not hold, since the ideals $(a_1),\ldots,(a_n)$ may be product-coprime in $\insprinc(D)$ without $a_1,\ldots,a_n$ being coprime (for example, take $a_1=y$, $a_2=y^2$ and $a_3=y^3$, where $y$ is a prime element of $D$). However, this can be repaired: $a_1,\ldots,a_n\in D^\nz$ are coprime if and only if the ideals $(a_1)^{k_1},\ldots,(a_n)^{k_n}$ are product-coprime in $\insprinc(D)$ for every $k_1,\ldots,k_n\inN$. The proof is essentially analogous to the one given for Proposition \ref{prop:prod-coprime}.

Proposition \ref{prop:full} carries over without significant changes: the maximal product-proper subsets of $\insprinc(D)$ are in bijective correspondence with the maximal ideals of $D$. Theorem \ref{teor:princ->rad} carries over as well: the only difference is that, instead of the restricted power set $\mathfrak{P}_\fin(\Max(D))$ it is necessary to use the free abelian group generated by $\Max(D)$.

In particular, this result directly implies that if $\insprinc(D)$ and $\insprinc(D')$ are isomorphic as partially ordered sets then the class groups $\Cl(D)$ and $\Cl(D')$ are isomorphic as groups, since the class group depends exactly on which ideals are principal. This result is also a consequence of the theory of monoid factorization (see \cite{ger-HK}), of which this reasoning can be seen as a more direct (but less general) version.
\end{oss}

\section{Calculating the rank}\label{sect:ordini}
The \emph{rank} $\rk G$ of an abelian group $G$ is the dimension of the tensor product $G\otimes\insQ$ as a vector space over $\insQ$. In particular, the rank of $G$ is $0$ if and only if $G$ is a torsion group; therefore, Corollary \ref{cor:rk0} can be rephrased by saying that, if $\monprinc(D)$ and $\monprinc(D')$ are order-isomorphic, then the rank of $\Cl(D)$ is $0$ if and only if the rank of $\Cl(D')$ is $0$. In this section, we want to generalize this result by showing that $\rk\Cl(D)$ is actually determined by $\monprinc(D)$ in every case.

Let $D$ be a Dedekind domain. If $\insid(D)$ is the set of proper ideals of $D$, then the quotient from $\insfracid(D)$ to $\Cl(D)$ restricts to a map $\pi:\insid(D)\longrightarrow\Cl(D)$, which is a monoid homomorphism (i.e., $\pi(IJ)=\pi(I)\cdot\pi(J)$). Moreover, $\pi$ is surjective since the class of $I$ coincide with the class of $dI$ for every $d\in D^\nz$.

There is also a natural map $\psi_0:\Cl(D)\longrightarrow\Cl(D)\otimes\insQ$, $g\mapsto g\otimes 1$, from the class group to the $\insQ$-vector space $\Cl(D)\otimes\insQ$; the map $\psi_0$ is a group homomorphism, and its kernel is the torsion subgroup $T$ of $\Cl(D)$. By construction, the image $\mathcal{C}$ of $\psi_0$ spans $\Cl(D)\otimes\insQ$ as a $\insQ$-vector space.

Thus, we have a chain of maps
\begin{equation*}
\insid(D)\xrightarrow{~~\pi~~}\Cl(D)\xrightarrow{~~\psi_0~~}\Cl(D)\otimes\insQ;
\end{equation*}
we denote by $\psi$ the composition $\psi_0\circ\pi$.

\begin{defin}\label{def:InvP}
Let $\Delta\subseteq\Max(D)$. A maximal ideal $Q$ is an \emph{almost inverse} of $\Delta$ if there is a set $\{P_1,\ldots,P_k\}\subseteq\Delta$ (not necessarily nonempty) such that $Q\wedge P_1\wedge\cdots\wedge P_n$ belongs to $\monprinc(D)$. We denote the set of almost inverses of $\Delta$ as $\Inv(\Delta)$.
\end{defin}

Our aim is to characterize $\Inv(\Delta)$ in terms of the map $\psi$; to do so, we use the terminology of \emph{ordered topological spaces} (for which we refer the reader to, e.g., \cite{davis-poslindep}). Given a $\insQ$-vector space $V$ and a set $S\subseteq V$, the \emph{positive cone} spanned by $S$ is 
\begin{equation*}
\pos(S):=\left\{\sum_{i=1}^k\lambda_i\mathbf{v}_i\mid \lambda_i\inQ^{\geq 0},\mathbf{v}_i\in S\right\};
\end{equation*}
if $C=\pos(S)$, we say that $C$ is \emph{positively spanned} by $S$. Symmetrically, the \emph{negative cone} is $\negt(S):=-\pos(S)$.
\begin{prop}\label{prop:InvP}
Let $\Delta\subseteq\Max(D)$. Then, $\Inv(\Delta)=\psi^{-1}(\negt(\psi(\Delta)))$.
\end{prop}
\begin{proof}
Let $Q\in\Inv(\Delta)$, and let $P_1,\ldots,P_n\in\Delta$ be such that $L:=Q\wedge P_1\wedge\cdots\wedge P_n\in\monprinc(D)$. Then, there is a principal ideal $I=aD$ with radical $L$; thus, there are positive integers $e,f_1,\ldots,f_n>0$ such that $I=Q^eP_1^{f_1}\cdots P_n^{f_n}$ (this holds also if $Q=P_i$ for some $i$). Since $I$ is principal, $\psi(I)=\mathbf{0}$; hence, 
\begin{equation*}
\mathbf{0}=\psi(I)=\psi(Q^eP_1^{f_1}\cdots P_n^{f_n})=e\psi(Q)+\sum_{i=1}^nf_i\psi(P_i).
\end{equation*}
Solving in $\psi(Q)$, we see that $\psi(Q)=\sum_i-\frac{f_i}{e}\psi(P_i)\in\negt(\psi(\Delta))$, as claimed.

Conversely, suppose that $\psi(Q)$ is in the negative cone. Then, either $\psi(Q)=\mathbf{0}$ (in which case $Q\in\Inv(\Delta)$ by taking no $P\in\Delta$ in the definition) or we can find $P_1,\ldots,P_n\in\Delta$ and negative rational numbers $q_1,\ldots,q_n$ such that $\psi(Q)=\sum_iq_i\psi(P_i)$. By multiplying for the minimum common multiple of the denominators of the $q_i$ we obtain a relation $e\psi(Q)+\sum_if_i\psi(P_i)=\mathbf{0}$, with $e,f_i\inN^+$. If $I:=Q^eP_1^{f_1}\cdots P_n^{f_n}$, it follows that $\pi(I)$ is torsion in the class group, i.e., there is an $n>0$ such that $I^n$ is principal; thus $\rad(I^n)=\rad(I)=Q\wedge P_1\wedge\cdots\wedge P_n\in\monprinc(D)$, as claimed.
\end{proof}

\begin{cor}\label{cor:posspan}
Let $\Delta\subseteq\Max(D)$. Then, $\Inv(\Delta)=\Max(D)$ if and only if $\psi(\Delta)$ positively spans $\Cl(D)\otimes\insQ$.
\end{cor}
\begin{proof}
Suppose $\Inv(\Delta)=\Max(D)$, and let $\mathbf{q}\in\Cl(D)\otimes\insQ$. Since the image $\mathcal{C}$ of $\psi$ generates $\Cl(D)\otimes\insQ$ as a $\insQ$-vector space and is a subgroup, there is a $d\inN^+$ such that $d\mathbf{q}\in\mathcal{C}$. Hence, $d\mathbf{q}=\psi(I)$ for some $I\in\insid(D)$; factorize $I$ as $P_1^{e_1}\cdots P_n^{e_n}$, with $P_i\in\Max(D)$ and $e_i>0$. By Proposition \ref{prop:InvP}, we have
\begin{equation*}
\psi(I)=\sum_ie_i\psi(P_i)\in\sum_ie_i\negt(\psi(\Delta))=\negt(\psi(\Delta)),
\end{equation*}
and thus also $\mathbf{q}=\inv{d}\psi(I)\in\negt(\psi(\Delta))$. Hence, $\psi(\Delta)$ negatively spans $\Cl(D)\otimes\insQ$, and thus it also positively spans $\Cl(D)\otimes\insQ$.

Conversely, suppose $\psi(\Delta)$ positively spans $\Cl(D)\otimes\insQ$; thus, it also negatively spans $\Cl(D)\otimes\insQ$. Let $Q\in\Max(D)$: then, $\psi(Q)\in\negt(\psi(\Delta))$, so that $Q\in\Inv(\Delta)$ by Proposition \ref{prop:InvP}. Hence, $\Inv(\Delta)=\Max(D)$.
\end{proof}

We can now characterize when the rank of $\Cl(D)$ is finite.
\begin{prop}\label{prop:finbasis}
Let $D$ be a Dedekind domain. Then, $\rk\Cl(D)<\infty$ if and only if there is a finite set $\Delta\subseteq\Max(D)$ such that $\Inv(\Delta)=\Max(D)$.
\end{prop}
\begin{proof}
Suppose first that $\rk\Cl(D)=n<\infty$. Then, $\Inv(\Max(D))=\Max(D)$, and thus $\psi(\Max(D))$ positively spans $\Cl(D)\otimes\insQ$, by Corollary \ref{cor:posspan}. Let $\{\mathbf{e}_1,\ldots,\mathbf{e}_n\}$ be a basis of $\Cl(D)\otimes\insQ$: then, each $\mathbf{e}_i$ belongs to the positive cone spanned by a finite subset $\Lambda_i$ of $\psi(\Max(D))$. Thus, the union $\Lambda$ of the $\Lambda_i$ is a finite set positively spanning $\Cl(D)\otimes\insQ$, so the corresponding subset $\Delta$ of $\Max(D)$ is finite and $\Inv(\Delta)=\Max(D)$ by Corollary \ref{cor:posspan}.

Conversely, suppose there is a finite set $\Delta=\{P_1,\ldots,P_k\}\subseteq\Max(D)$ such that $\Inv(\Delta)=\Max(D)$. For every $Q\in\Max(D)$, there are $i_1,\ldots,i_r$ such that $Q\wedge P_{i_1}\wedge\cdots\wedge P_{i_r}\in\monprinc(D)$; as in the proof of Proposition \ref{prop:InvP}, it follows that there are $e,f_1,\ldots,f_r>0$ such that $Q^eP_{i_1}^{f_1}\cdots P_{i_r}^{f_r}$ is principal. It follows that $[Q]\otimes 1$ belongs to the $\insQ$-vector subspace of $\Cl(D)\otimes\insQ$ generated by $P_{i_1}\otimes 1,\ldots,P_{i_r}\otimes 1$. Since $Q$ was arbitrary, the set $\{P_1\otimes 1,\ldots,P_k\otimes 1\}$ is a basis of $\Cl(D)\otimes\insQ$. In particular, $\rk\Cl(D)=\dim_\insQ\Cl(D)\otimes\insQ\leq k<\infty$.
\end{proof}

We will also need a criterion to understand when $\Inv(\Delta)$ correspond to a linear subspace.
\begin{prop}\label{prop:inv-linsub}
Let $\Delta\subseteq\Max(D)$. Then, $\negt(\psi(\Delta))$ is a linear subspace of $\Cl(D)\otimes\insQ$ if and only if $\Delta\subseteq\Inv(\Delta)$.
\end{prop}
\begin{proof}
Suppose $\negt(\psi(\Delta))$ is a linear subspace, and let $Q\in\Delta$. Then, there are $P_i\in\Delta$, $\lambda_i\inQ^-$ such that $\psi(Q)=\sum_i\lambda_i\psi(P_i)$; multiplying by the minimum common multiple of the denominators we get an equality $e\psi(Q)+\sum_if_i\psi(P_i)=\mathbf{0}$ where $e,f_i\inN^+$. Let $I:=Q^eP_1^{f_1}\cdots P_n^{f_n}$: then, $\psi(I)=\mathbf{0}$, so that $\pi(I)$ is torsion in $\Cl(D)$, i.e., $I^n$ is principal for some $n$. Thus, $Q\wedge P_1\wedge\cdots\wedge P_n\in\monprinc(D)$, and $Q\in\Inv(\Delta)$.

Conversely, suppose $\Delta\subseteq\Inv(\Delta)$, and let $\mathbf{q}$ be an element of the linear subspace generated by $\psi(\Delta)$. Then, there are $P_i,Q_j\in\Delta$, $\theta_i\inQ^+$ and $\mu_j\in\insQ^-$ such that
\begin{equation*}
\mathbf{q}=\sum_i\theta_i\psi(P_i)+\sum_i\mu_j\psi(Q_j).
\end{equation*}
By construction, each $\theta_i\psi(P_i)$ belongs to $\pos(\psi(\Delta))$. Furthermore, each $\psi(Q_j)$ is in $\negt(\psi(\Delta))$ by Proposition \ref{prop:InvP}, and thus $\mu_j\psi(Q_j)\in\pos(\psi(\Delta))$ for every $j$. Therefore, $\mathbf{q}\in\pos(\psi(\Delta))$, so the positive cone of $\psi(\Delta)$ is a linear subspace and $\negt(\psi(\Delta))=\pos(\psi(\Delta))$ is a subspace too.
\end{proof}

Proposition \ref{prop:finbasis} can be interpreted by saying that $\rk\Cl(D)$ is finite if and only if $\Max(D)$ is ``negatively generated'' by a finite set. In the case of finite rank, we need a way to link the dimension of $\Cl(D)\otimes\insQ$ with the cardinality of the sets spanning it as a positive cone; that is, we need to consider a notion analogue to the basis of a vector space.

Since we need only to consider the case of finite rank, from now on we suppose that $n:=\rk\Cl(D)<\infty$, and we identify $\Cl(D)\otimes\insQ$ with $\insQ^n$.
\begin{defin}
A set $X\subseteq\insQ^n$ is \emph{positive basis} of $\insQ^n$ if $\pos(X)=\insQ^n$ and if $\pos(X\setminus\{x\})\neq\insQ^n$ for every $x\in X$.
\end{defin}

\begin{defin}
A subset $\Delta\subseteq\Max(D)$ is an \emph{inverse basis} of $\Max(D)$ if $\Inv(\Delta)=\Max(D)$ and $\Inv(\Delta')\neq\Max(D)$ for every $\Delta'\subsetneq\Delta$.
\end{defin}

These two notions are naturally connected.
\begin{prop}\label{prop:invbasis}
Let $\Delta\subseteq\Max(D)$. Then, $\Delta$ is an inverse basis of $\Max(D)$ if and only if $\psi(\Delta)$ is a positive basis of $\insQ^n$.
\end{prop}
\begin{proof}
If $\Delta$ is an inverse basis, then $\psi(\Delta)$ positively spans $\insQ^n$ by Corollary \ref{cor:posspan}, while $\psi(\Delta')$ does not for every $\Delta'\subsetneq\Delta$ (again by the corollary). Hence, $\psi(\Delta)$ is a positive basis. The converse follows in the same way.
\end{proof}

Given a positive basis $X$ of $\insQ^n$, we call a partition $\{X_1,\ldots,X_s\}$ of $X$ a \emph{weak Reay partition} if, for every $j$, the positive cone of $X_1\cup\cdots\cup X_i$ is a linear subspace of $\insQ^n$. The following is a variant of \cite[Theorem 2]{reay65}.
\begin{prop}\label{prop:reaypart}
Let $X$ be a positive basis of $\insQ^n$. Then:
\begin{enumerate}[(a)]
\item\label{prop:reaypart:leq} every weak Reay partition of $X$ has cardinality at most $|X|-n$;
\item\label{prop:reaypart:ex} there is a weak Reay partition of cardinality $|X|-n$.
\end{enumerate}
\end{prop}
\begin{proof}
Let $\{X_1,\ldots,X_s\}$ be a weak Reay partition, and let $V_i$ be the linear space spanned by $X_1,\ldots,X_i$ (with $V_0:=(0)$). We claim that $\dim V_i-\dim V_{i-1}\leq |X_i|-1$. Indeed, let $X_i:=\{z_1,\ldots,z_t\}$: then, $-z_t$ belongs to the positive cone generated by $V_{i-1}$ and $X_i$, and thus we can write $-z_t=y+\sum_j\lambda_jz_j$ for some $y\in V_{i-1}$ and $\lambda_j\geq 0$. Thus, $-(1+\lambda_t)z_t=y+\lambda_1z_1+\cdots+\lambda_{t-1}z_{t-1}$, and since $\lambda_t\neq-1$ we have that $z_t$ is linearly dependent from $X_1\cup\cdots\cup X_{i-1}\cup\{z_1,\ldots,z_{t-1}\}$. Hence, $\dim V_i\leq\dim V_{i-1}+t-1$, as claimed.

Therefore, 
\begin{equation*}
\begin{aligned}
n=\dim\insQ^n= & (\dim V_s-\dim V_{s-1})+\cdots+\dim V_1\leq\\
\leq & (|X_s|-1)+\cdots+(|X_1|-1)=|X|-s,
\end{aligned}
\end{equation*}
and thus $s\leq |X|-n$, and \ref{prop:reaypart:leq} is proved. \ref{prop:reaypart:ex} is a direct consequence of \cite[Theorem 2]{reay65}.
\end{proof}

Similarly, if $\Delta\subseteq\Max(D)$ is an inverse basis of $\Max(D)$, we call a partition $\{\Delta_1,\ldots,\Delta_s\}$ a \emph{weak Reay partition} if $\Delta_1\cup\cdots\cup\Delta_i\subseteq\Inv(\Delta_1\cup\cdots\cup\Delta_i)$ for every $i$. 
\begin{prop}\label{prop:corrispreay}
Let $\Delta\subseteq\Max(D)$ be an inverse basis of $\Max(D)$, and let $\{\Delta_1,\ldots,\Delta_s\}$ be a partition of $\Delta$. Then, $\{\Delta_1,\ldots,\Delta_s\}$ is a weak Reay partition of $\Delta$ if and only if $\{\psi(\Delta_1),\ldots,\psi(\Delta_s)\}$ is a weak Reay partition of $\psi(\Delta)$.
\end{prop}
\begin{proof}
By Proposition \ref{prop:inv-linsub}, $\Delta_1\cup\cdots\cup\Delta_i\subseteq\Inv(\Delta_1\cup\cdots\cup\Delta_i)$ if and only if the positive cone of $\psi(\Delta_1\cup\cdots\cup\Delta_i)=\psi(\Delta_1)\cup\cdots\cup\psi(\Delta_i)$ is a linear subspace of $\insQ^n$. The claim now follows from the definition.
\end{proof}

\begin{teor}\label{teor:rank}
Let $D,D'$ be Dedekind domains such that $\monprinc(D)$ and $\monprinc(D')$ are isomorphic. Then, $\rk\Cl(D)=\rk\Cl(D')$.
\end{teor}
\begin{proof}
Let $\phi:\monprinc(D)\longrightarrow\monprinc(D')$ be an isomorphism; by Theorem \ref{teor:princ->rad}, we can find an isomorphism $\Phi:\monrad(D)\longrightarrow\monrad(D')$ sending $\monprinc(D)$ to $\monprinc(D')$. In particular, $\Phi(\Max(D))=\Max(D')$.

Since $\Inv(\Delta)$ is defined only through $\monprinc(D)$ and $\monrad(D)$, $\Phi$ respects the inverse construction, in the sense that $\Phi(\Inv(\Delta))=\Inv(\Phi(\Delta))$ for every $\Delta\subseteq\Max(D)$. In particular, $\Inv(\Delta)=\Max(D)$ if and only if $\Inv(\Phi(\Delta))=\Max(D')$; by Proposition \ref{prop:finbasis}, it follows that $\rk\Cl(D)=\infty$ if and only if $\rk\Cl(D')=\infty$. 

Suppose now that the two ranks are finite, say equal to $n$ and $n'$ respectively. Let $\Delta\subseteq\Max(D)$ be an inverse basis of $\Max(D)$. Let $\{\Delta_1,\ldots,\Delta_s\}$ be a weak Reay partition of $\Delta$ of maximum cardinality; by Propositions \ref{prop:corrispreay} and \ref{prop:reaypart}, $s=|\Delta|-n$.

Every weak Reay partition of $\Delta$ gets mapped by $\Phi$ into a weak Reay partition of $\Delta':=\psi(\Delta)$, and conversely; therefore, the maximum cardinality of the weak Reay partitions of $\Delta'$ is again $|\Delta|-n$. However, applying Propositions \ref{prop:corrispreay} and \ref{prop:reaypart} to $\Delta'$ we see that this quantity is $|\Delta'|-n'$; since $|\Delta|=|\Delta'|$, we get $n=n'$, as claimed.
\end{proof}

\begin{cor}\label{cor:rank-fg}
Let $D,D'$ be Dedekind domains, and let $T(D)$ (respectively, $T(D')$) be the torsion subgroup of $\Cl(D)$ (resp., $\Cl(D')$). If $\monprinc(D)$ and $\monprinc(D')$ are isomorphic and if $\Cl(D)$ and $\Cl(D')$ are finitely generated, then $\Cl(D)/T(D)\simeq\Cl(D')/T(D')$.
\end{cor}
\begin{proof}
Since $\Cl(D)$ is finitely generated, it has finite rank $n$ and $\Cl(D)/T(D)\simeq\insZ^n$; analogously, $\Cl(D')/T(D')\simeq\insZ^m$, where $m:=\rk\Cl(D')$. By Theorem \ref{teor:rank}, $n=m$, and in particular $\Cl(D)/T(D)\simeq\Cl(D')/T(D')$.
\end{proof}

\section{Counterexamples}\label{sect:counterex}
In this section, we collect some examples showing that Theorem \ref{teor:rank} is, in many ways, the best possible.

\begin{ex}
It is not possible to improve the conclusion of Theorem \ref{teor:rank} from ``$\rk\Cl(D)=\rk\Cl(D')$'' to ``$\Cl(D)\simeq\Cl(D')$''. Indeed, if $\rk\Cl(D)=0$ (i.e., if $\Cl(D)$ is torsion) then $\monprinc(D)=\monrad(D)$, and thus whenever $\rk\Cl(D)=\rk\Cl(D')=0$ the posets $\monprinc(D)$ and $\monprinc(D')$ are isomorphic.
\end{ex}

For the next examples, we need to use a construction of Claborn \cite{claborn-relations}.

Let $G:=\sum_ix_i\insZ$ be the free abelian group on the countable set $\{x_i\}_{i\inN}$. Let $I$ be a subset of $G$ satisfying the following two properties:
\begin{itemize}
\item all coefficients of the elements of $I$ (with respect to the $x_i$) are nonnegative;
\item for every finite set $x_{i_1},\ldots,x_{i_k}$ and every $n_1,\ldots,n_k\inN$ there is an element $y$ of $I$ such that the component of $y$ relative to $x_{i_t}$ is $n_t$.
\end{itemize}
Then, \cite[Theorem 2.1]{claborn-relations} says that there is an integral domain $D$ with countably many maximal ideals $\{P_i\}_{i\inN}$ such that the map sending the ideal $P_1^{n_1}\cdots P_k^{n_k}$ to $n_1x_1+\cdots+n_kx_k$ sends principal ideals to elements of the subgroup $H$ generated by $I$. In particular, $\Cl(D)\simeq G/H$.

\begin{ex}
Corollary \ref{cor:rank-fg} does not hold without the hypothesis that $\Cl(D)$ and $\Cl(D')$ are finitely generated.

For example, let $H_1$ be the subgroup of $G$ generated by $x_n+x_{n+1}$, as $n$ ranges in $\insN$, and $I_1$ to be the subset of the elements of $H_1$ having all coefficients nonnegative. Then, $I_1$ satisfies the above conditions; the corresponding domain $D_1$ has a class group isomorphic to $\insZ$, and its prime ideals are concentrated in two classes: if $n$ is even $P_n$ is equivalent to $P_0$, if $n$ is odd $P_n$ is equivalent to $P_1$, and $P_0P_1$ is principal. (This is exactly Example 3-2 of \cite{claborn-relations}.) In particular, $\monprinc(D_1)$ is equal to the members of $\monrad(D_1)$ that are contained both in some $P_n$ with $n$ even and in some $P_m$ with $m$ odd.

Let now $H_2$ to be the subgroup of $G$ generated by $x_n+2x_{n+1}$, as $n$ ranges in $\insN$, and let $I_2$ be the subset of the elements of $H_2$ having all coefficients nonnegative. Then, $I_2$ too satisfies the condition above. Let $D_2$ be the corresponding Dedekind domain. Then, $\Cl(D_2)$ is isomorphic to the quotient $G/H_2$, which is isomorphic to the subgroup $\insZ(2^\infty)$ of $\insQ$ generated by $1,\inv{2},\inv{4},\ldots,\inv{2^n},\ldots$ (that is, to the Pr\"ufer $2$-group): this can be seen by noting that the map
\begin{equation*}
\begin{aligned}
G & \longrightarrow \insQ,\\
P_n & \longmapsto (-1)^n\inv{2^n}
\end{aligned}
\end{equation*}
is a group homomorphism with kernel $H_2$ and range $\insZ(2^\infty)$. In this isomorphism, the prime ideals $Q_n$ with $n$ even are mapped to positive elements of $\insZ(2^\infty)$, while the prime ideals $Q_m$ with $m$ odd are mapped to the negative elements. Hence, $\monprinc(D_2)$ is equal to the member of $\monrad(D_2)$ that are contained in both an ``even'' and an ``odd'' prime. 

Therefore, the map $\monrad(D_1)\longrightarrow\monrad(D_2)$ sending $P_{i_1}\cap\cdots\cap P_{i_k}$ to $Q_{i_1}\cap\cdots\cap Q_{i_k}$ is an isomorphism sending $\monprinc(D_1)$ to $\monprinc(D_2)$. However, the class groups of $D_1$ and $D_2$ are both torsionfree (i.e., $T(D_1)=T(D_2)=0$) but not isomorphic.
\end{ex}

\begin{ex}
The converse of Theorem \ref{teor:rank} does not hold; that is, it is possible that $\rk\Cl(D)=\rk\Cl(D')$ even if $\monprinc(D)$ and $\monprinc(D')$ are not isomorphic.

Take $H_1$ and $D_1$ as in the previous example.

Take $H_3$ to the the subgroup of $G$ generated by $x_0$ and by $x_n+x_{n+1}$ for $n>0$, and let $I_3$ be the subset of the elements of $H_3$ having all coefficients nonnegative. Then, $I_3$ satisfies Claborn's conditions, and the corresponding domain $D_3$ satisfies $\Cl(D_3)\simeq\insZ$ (in particular, $\rk\Cl(D_3)=1$), so $\Cl(D_1)$ and $\Cl(D_3)$ are isomorphic. 

However, $D_3$ has a principal maximal ideal (the one corresponding to $x_0$), while $D_1$ does not. Therefore, there is no isomorphism $\monrad(D_1)\longrightarrow\monrad(D_3)$ sending $\monprinc(D_1)$ to $\monprinc(D_3)$; by Theorem \ref{teor:princ->rad}, it follows that $\monprinc(D_1)$ and $\monprinc(D_3)$ cannot be isomorphic.
\end{ex}

\bibliographystyle{plain}
\bibliography{/bib/articoli,/bib/libri,/bib/miei}
\end{document}